\documentclass{article}

\usepackage{graphicx}
\usepackage{amsmath,amssymb,amsthm}
\usepackage{caption,color}
\usepackage{subcaption}
\usepackage{graphicx}
\usepackage{float}
\usepackage{a4wide}
\usepackage{epstopdf}

\newtheorem{theorem}{Theorem}
\newtheorem{proposition}[theorem]{Proposition}

\newtheorem{remark}[theorem]{Remark}
\usepackage{float}

\begin{document}
	
\title{Pest control using farming awareness: 
impact of time delays and optimal use of biopesticides\thanks{This is a preprint of a paper 
whose final and definite form is published by 'Chaos Solitons Fractals' (ISSN: 0960-0779). 
Paper Submitted 08/Sept/2020; Revised 25/Feb/2021; Accepted 10/March/2021.}}
	
\author{Teklebirhan Abraha$^{1}$\\ {\tt tekbir98@yahoo.com}
\and Fahad Al Basir$^{2}$\\ {\tt fahadbasir@gmail.com}
\and Legesse Lemecha Obsu$^{1}$\\ {\tt legesse.lemecha@astu.edu.et}
\and Delfim F. M. Torres$^{3}$\thanks{Corresponding author. Email: delfim@ua.pt}\\ {\tt delfim@ua.pt}}
	
\date{$^{1}$\small{Department of Mathematics, Adama Science and Technology University, Adama, Ethiopia}\\[0.2cm]
$^{2}$\small{Department of Mathematics, Asansol Girls' College, West Bengal 713304, India}\\[0.2cm]
$^{3}$\small{R\&D Unit CIDMA, Department of Mathematics, University of Aveiro, 3810-193 Aveiro, Portugal}}
	
\maketitle
	
	
\begin{abstract}
We investigate a mathematical model in crop pest management,
considering plant biomass, pest, and the effect of farming awareness.
The pest population is divided into two compartments: susceptible pest
and infected pest. We assume that the growth rate of self-aware people
is proportional to the density of healthy pests present in the crop field.
Impacts of awareness is modeled via a saturated term. It is further assumed
that self-aware people will adopt biological control methods, namely integrated
pest management. Susceptible pests are detrimental to crops and, moreover,
there may be some time delay in measuring the healthy pests in the crop field.
A time delay may also take place while becoming aware of the control strategies
or taking necessary steps to control the pest attack. In agreement, we develop
our model incorporating two time delays into the system. The existence and the stability
criteria of the equilibria are obtained in terms of the basic reproduction number
and time delays. Stability switches occur through Hopf-bifurcation when time delays
cross critical values. Optimal control theory has been applied for the cost-effectiveness
of the delayed system. Numerical simulations illustrate the obtained analytical results.
		
\bigskip
		

\bigskip

\end{abstract}
	
	
\section{Introduction}
	
In recent times, integrated pest management is gaining more attention
among researchers and its application is also increasing in the crop field.
This method seeks to reduce the reliance on pesticides by emphasizing
the contribution of biological control agents. The important role of microbial pesticides
in integrated pest management is well-known in agriculture, forestry, and public health.
As integrated pest management, bio-pesticides give noticeable pest control reliability
in case of crops \cite{Bhatta2}. The use of viruses against insect pests, as pest control agents,
is seen in North America and in European countries \cite{Franz,Falcon,Naranjo}.
Awareness campaigns, in particular through radio or TV, are required so that people
will gain trustworthiness on a biological control approach.
Farmers, in their own awareness, can keep the crop under observation
and, therefore, if correctly instructed, they will spray bio-pesticides
or incorporate fertile to make the pest susceptible to their bio-agents.
For the effect of awareness coverage in controlling infectious diseases
we refer the reader to \cite{AMC}, where a SIS model is formulated considering 
individuals' behavioral changes due to the influences of media coverage, 
and where the susceptible class is divided into two subclasses: 
aware susceptible and unaware susceptible. 

Correct and relevant knowledge about crop and its pests is very much essential
for people engaged in cultivation. The role of electronic media is critical
for keeping the farming community updated and by providing them with
relevant agricultural information \cite{Khan,Kumar2}. Accessible pesticide
information campaigns help farmers to be aware on the serious risks that pesticides
have on human health and environment and minimize negative effects \cite{risks}.
Adopting awareness programs, intended to educate farmers', results
in a better comprehensive development for the cultivars and also for the farmers.
Farmers learn the use and dangers of pesticides mainly by oral communication.
Self-aware farmers employ considerably improved agronomic practices,
safeguarding health and reducing environmental hazards \cite{Yang2014}.
Therefore, awareness is important in crop pest management.
	
Television, radio, and mobile telephony are particularly useful media
in providing information about agricultural practices and crop protection \cite{Khan}.
Adopting new technologies during agricultural awareness programs represents
a major route for innovating and improving agronomy. Le Bellec et al. \cite{LeBellec}
have studied how an enabling environment for interactions between farmers,
researchers, and other factors, can contribute to reduce current problems
associated with crop. Al Basir et al., describe the participation
of farming communities in Jatropha projects for biodiesel production
and protection of plants from mosaic disease, using a mathematical model
to forecast the development of renewable energy resources \cite{mma}.
In \cite{EC2}, authors have developed a mathematical model for pest control
using bio-pesticides. Moreover, they also incorporate optimal control theory
to minimize the cost in pest management due to bio-pesticides.
For the use of optimal control theory to eradicate the number of parasites
in agroecosystems, see \cite{MR3660471}. The usefulness of time-delays
in epidemiological modeling is well-known \cite{MyID:430}. In \cite{JTB}, 
a model for pest control is proposed and analyzed, 
where the impact of farming awareness
and a time delay in local awareness is investigated. They conclude
that raising awareness among people, with tolerable time delay,
may be a proper aspect for the control of pests in a crop field
while reducing the serious issues that pesticides have on human
health and environment \cite{JTB}. Later, in \cite{IJB}, Al Basir
has discussed the effects of delay in pest control due to the implementation
of control interventions. In \cite{ridm}, Al Basir and Ray analyze
the dynamics of vector borne plant disease dynamics influenced
by farming awareness. Here, a mathematical model is formulated
to protect crops through awareness campaigns, modeled via saturated terms,
and a delayed optimal control problem for biopesticides is posed and solved.
	
The paper is organized as follows. In Section~\ref{sec:2},
the model is derived assuming that the rate of awareness
is proportional to the number of susceptible pests in the field.
Moreover, we make the model more realistic considering
a time delay due to the measure of pest in the field.
In Section~\ref{sec:3}, nonnegativity and boundedness
of the solutions are proved by finding the invariant region.
The equilibria, the basic reproduction number,
and the stability of the pest-free equilibrium
are studied, using qualitative theory, in Section~\ref{sec:4}.
In Sections~\ref{sec:5a} and \ref{sec4.2}, we study the direction 
and stability of the Hopf bifurcation, investigating the stability switches 
of the equilibrium points, respectively for the system without and with delays.
We show that stability switch occurs through Hopf bifurcation.
An optimal control problem is then formulated and solved analytically
in Section~\ref{sec:5}, with the goal to minimize the cost of biopesticides.
The obtained analytical results are illustrated through numerical simulations
in Section~\ref{sec:6}. We end with Section~\ref{sec:7} of
discussion and conclusion.
	
	
\section{Model Derivation}
\label{sec:2}
	
In this section, we present a deterministic pest control model to prevent yield loss.
The model includes crop biomass, $X(t)$; healthy pest, $S(t)$; infected pest, $I(t)$;
and level of awareness, $A(t)$. Due to the finite size of the crop field,
we assume logistic growth for the density of crop biomass, with net growth
rate $r$ and carrying capacity $K$. Susceptible attacks the crop, thereby
causing considerable crop reduction. If we infect the susceptible pest by pesticides,
then the attack by pest can be controlled. Here we assume that self-aware farmers
will adopt biological pesticides for the control of crop pest as it has fewer
side effects and are also environment-friendly \cite{JTB}. Bio-pesticides
are used to infect healthy pest. The infected pest has additional mortality
due to infection. We further assume that infected pests cannot consume crop \cite{EC2}.
	
Let $\lambda$ be the consumption rate of pests with conversion rate $m$.
There is a pest infection rate, $h$, because of self-aware human interactions
and activity such as the use of bio-pesticides, modeled via the usual mass action term
$\frac{hAS}{1+A}$ \cite{EC1}. Here, $d$ is the natural mortality rate of pest
and $\alpha$ is the additional mortality rate of infected pest due to
self-aware people activity.
It is assumed that the level of awareness will increase at a rate proportional
to the number of susceptible pests per plant noticed in the farming system.
A fading of interest in this exploitation is possible,
and we denote by $\eta$ the rate of awareness fading.
	
There might be a delay in observing the number or activity of pests
in a field. Usually, this estimate is generated by observing previous
cases of pest occurrence and thus the intensity of awareness
and the level of implication of preventive countermeasures varies.
A delay $\tau_1> 0$ is expected in the execution of such measures.
The intensity of the awareness programs, being executed at time $t$,
will be in accordance with the number of pests at the time $t-\tau_1> 0$.
Also, after an advertisement, farmers take some time to become aware
of the technologies/pesticide to use and their management. We assume
$\tau_2>0$ as the time delay parameter taken for organizing an awareness
campaign and the farmers to become aware.
	
Based on the above assumptions, the mathematical model
\begin{equation}
\label{delay1}
\begin{cases}
\frac{dX(t)}{dt} = r X(t)\left[1 - \frac{X(t)}{K}\right] - \lambda X(t)S(t),\\
\frac{dS(t)}{dt} = m\lambda X(t)S(t) - \frac{hA(t-\tau_2)S(t)}{1+A(t-\tau_2)}  -dS(t),\\
\frac{dI(t)}{dt} = \frac{hA(t-\tau_2)S(t)}{1+A(t-\tau_2)} - (d+\alpha)I(t),\\
\frac{dA(t)}{dt} = \omega + aS(t-\tau_1) - \eta A(t),
\end{cases}
\end{equation}
is formulated subject to initial conditions
\begin{equation}
\label{eq:ic}
X(\phi)>0, \quad S(\phi)>0, \quad I(\phi)>0, \quad A(\phi)>0,
\end{equation}
where $\phi\in[-\tau, 0]$ with $\tau=\max\{\tau_1,\tau_2\}$.
The meaning of the parameters of model \eqref{delay1} is summarized
in Table~\ref{tab3.1}, together with the values that later
will be used in our numerical simulations.
\begin{table}[H]
\begin{center}
\caption{Meaning of the parameters of model \eqref{delay1}
with values used in Section~\ref{sec:6} for numerical simulations.}
{\small
\begin{tabular}{|c|l|l|} \hline
Parameters  & Description  & Value \\ \hline
$r$ & growth rate of crop biomass & 0.1 day$^{-1}$  \\
$K$ & maximum density of crop biomass & 100 m$^{-2}$ \\
$\lambda$ & consumption rate of pest& 0.05 biomass$^{-1}$ day$^{-1}$\\
$d$ & natural mortality of pest& 0.02 day$^{-1}$\\
$m$ & conversion efficacy & 0.6 \\
$h$ & aware people activity rate & 0.025 day$^{-1}$\\
$a$ & aware people growth rate &  0.012 {day$^{-1}$}  \\
$\alpha$ & additional mortality rate & 0.025 {day$^{-1}$}\\
$\omega$ & rate of awareness from global source & 0.003 {day$^{-1}$}\\
$\eta$ & fading of memory of aware people & 0.015 {day$^{-1}$} \\ \hline
\end{tabular}}
\label{tab3.1}
\end{center}
\end{table}
	
	
\section{Model Analysis}
\label{sec:3}
	
In this section, some basic properties of the solutions
of the delayed system (\ref{delay1}) are proved. In concrete,
we show positive invariance and boundedness of solutions.

\begin{theorem}[Non-negative invariance]
\label{thm1}
All solutions of (\ref{delay1}) with given/fixed initial
conditions \eqref{eq:ic} are non-negative.
\end{theorem}
	
\begin{proof}
The result is straightforward: using the fundamental lemma in \cite{57},
one can easily show that a solution to the initial value problem
(\ref{delay1})--(\ref{eq:ic}) exists in the region $\mathbb{R}^4_+$
and all solutions remain non-negative for all $t > 0$.
Therefore, the positive octant $\mathbb{R}^4_+$ is an invariant region.
\end{proof}
	
Theorem~\ref{thm1} is important because positivity implies,
biologically, the survival of the populations. We now prove
another important characteristic of the solutions of (\ref{delay1}):
they are bounded.
	
\begin{theorem}[Boundedness of solutions]
\label{thm2}
Every solution of system (\ref{delay1}) that starts in
\begin{equation*}
\mathcal{D} = \left\{(X, S, I, A) \in \mathbb{R}^4_+
: 0\leq X\leq B_1, 0\leq S+I+X\leq B_2,0\leq A\leq\frac{\omega+aM}{\eta}\right\}
\end{equation*}
is uniformly bounded, where $\mathcal{D} $ is defined
with $B_1=\max~\{X(0),K\}$ and $B_2=\frac{aM (r+4d)}{4d}$.
\end{theorem}
	
\begin{proof}
Let us consider the first equation of our model \eqref{delay1}. Then,
\begin{equation}
\label{eq24}
\frac {dX}{dt} = r X\left[ 1-\frac{X}{K}\right]
-\lambda X S\leq rX\left[1-\frac{X}{K}\right]
\Longrightarrow \underset {t\rightarrow \infty}{ \lim~\sup}~X\leq M,
\end{equation}
where $M=\max~\{X(0), K\}$. Let $W=X+S+I$ at any time $t$. It follows,
using (\ref{eq24}) and the fact that $r X\left[ 1-\frac{X}{K}\right]$
is quadratic in $X$ and its maximum value is $\frac{rk}{4}$, that
\begin{eqnarray*}
\frac {dW}{dt} &=& r X\left[ 1-\frac{X}{K}\right]-dS-(d+\alpha)I\\
&\leq & r X\left[ 1-\frac{X}{K}\right]+dX-d(S+I+X)\\
&\leq &\frac{rM}{4}+dM-dW.
\end{eqnarray*}
Now, after a simple calculation, one gets
\begin{equation}
\label{eq25}
\underset {t\rightarrow \infty}{ \lim~ \sup }~W
=\frac{\frac{rM}{4}+dM}{d}=\frac{M (r+4d)}{4d}.
\end{equation}
Finally, from the last equation of system \eqref{delay1} and (\ref{eq25}),
we get
\begin{equation*}
\frac{dA}{dt} = aS(t-\tau) - \eta A\leq \frac{aM (r+4d)}{4d}-\eta A,
\end{equation*}	
which implies that $\underset {t\rightarrow \infty}{ \lim~ \sup }~A=\frac{aM (r+4d)}{4\eta d}$.
Thus, the region of attraction given by the set $\mathcal{D}$ is positively invariant,
attracting all solutions initiating inside the interior of the positive octant.
\end{proof}
	
	
\section{Equilibria and Stability}
\label{sec:4}
	
Model (\ref{delay1}) has three equilibria:
(i) the origin $E_0=\left(0,0,0,\frac{\omega}{\eta}\right)$,
(ii) the equilibrium in which only the healthy plants population thrives,
$E_1=\left(K,0,0,\frac{\omega}{\eta}\right)$, which is always feasible,
and (iii) the coexistence equilibrium, $E^*=(X^*,S^*,I^*,A^*)$ with
\begin{equation}
\label{eq:E*}
X^*=\frac{hA^*+d(1+A^*)}{m\lambda(1+A)},
\quad S^*=\frac{\eta A^*-\omega}{a},
\quad I^*=\frac{h S^*A^*}{(1+A)(a+\alpha)},
\end{equation}
where $A^*$ is the positive root of
\begin{equation}
\label{endemeq}
L_1A^2+L_2A+L_3=0
\end{equation} 
with
$$
L_1=-m\lambda^2\eta<0,
\quad L_2=-ra[h+d]+m\lambda^2[\omega-\eta],
\quad L_3=ra[m\lambda K-d]+m\lambda^2\omega.
$$

Our next result characterizes the
feasibility of the coexistence equilibrium $E^*$.

\begin{proposition}[Feasibility of the coexistence equilibrium]
\label{prop_1}
Let $E^*=(X^*,S^*,I^*,A^*)$ be as \eqref{eq:E*}.
\begin{itemize}
\item[(i)] If $L_3>0$, then there is 
a unique coexisting equilibrium $E^*$.

\item[(ii)] If $L_3=0$ and $L_2>0$, then there 
is a unique coexisting equilibrium $E^*$ 
with $A^*=-\frac{L_2}{L_1}$.

\item[(iii)] If $L_3<0$ and $L_2\leq 0$, 
then there is no positive coexisting equilibrium of the system.

\item[(iv)] Assuming that $L_3<0$ and $L_2>0$, 
\begin{itemize}
\item if $D=A^2_2-4 L_1 L_3>0$, then there exist 
two coexisting equilibria $E^*$ with $V^*=\frac{-L_2\pm\sqrt{D}}{2L_1}$;

\item if $D=0$, then there exists a unique coexisting equilibrium $E^*$;

\item if $D<0$, then there exists no positive coexisting equilibrium $E^*$.
\end{itemize}
\end{itemize}
\end{proposition}

\begin{proof}
The result follows applying Descartes' rule of signs to equation \eqref{endemeq}. 
\end{proof}

Linearizing system (\ref{delay1}) about 
the pest free equilibrium $E_1$, we obtain that
\begin{equation}
\label{2}
\frac{dY}{dt}=F Y(t)+GY(t-\tau_1)+HY(t-\tau_2)
\end{equation}
with $F$, $G$ and $H$ the $4\times4$ matrices 
$$
F=\left[
\begin{array}{cccc}
F_{11}& -\lambda X & 0 & 0\\
\lambda S &  0 & 0 & 0 \\
0 & \frac{hA}{1+A} & -d - \alpha& 0\\
0 & 0 & 0 &  -\eta\\
\end{array}
\right],
\
G=\left[
\begin{array}{cccc}
0 & 0 & 0 & 0\\
0 & 0 & 0 & \frac{h S}{(1+A)^2} \\
0 & 0 & 0 & \frac{h S}{(1+A)^2} \\
0 & 0 & 0 & 0
\end{array}
\right],
\
H=\left[
\begin{array}{cccc}
0 & 0 & 0 & 0\\
0 & 0 & 0 & 0 \\
0 & 0 & 0 & 0 \\
0 & a & 0 & 0
\end{array}
\right],
$$
where $F_{11}:=r\left(1- \frac{2X}{K}\right) -\lambda S$.
The characteristic equation is given by
\begin{equation*}
\triangle(\xi) =\mid {\xi}{I}-F-e^{-{\xi}{\tau_1}}G-e^{-{\xi}{\tau_2}}H\mid=0,
\end{equation*}
that is,
\begin{equation*}
(\xi+d+\alpha)(\xi^3+a_1\xi^2+a_2\xi+a_5+e^{-{\xi}{(\tau_1+\tau_2)}}[a_3+a_4\xi])=0,
\end{equation*}
with one eigenvalue being $\xi = -(d+\alpha)<0$ and the rest of the spectrum
being given by the roots of the transcendental equation
\begin{equation}
\label{3}
\psi(\xi,\tau_1,\tau_2)
=\xi^3+a_1\xi^2+a_2\xi+[a_3+ a_4\xi]e^{-\xi(\tau_1+\tau_2)} +a_5 =0,
\end{equation}
where
\begin{equation}
\label{eq:val:a:i}
a_1:=\eta-F_{11},
\quad a_2:=\eta F_{11}+\lambda^2 XS,
\quad a_3:=-\frac{ahSF_{11}}{(1+A)^2},
\quad a_4:=\frac{ahS}{(1+A)^2},
\quad a_5:=\eta\lambda^2XS.
\end{equation}
Note that the origin, $E_0$, is always stable, in view of the fact that its eigenvalues
$r>0$, $-m<0$, $-\zeta <0$, and $-\eta$ are all negative. At the equilibrium
$E_1$ with only healthy plants, however, the eigenvalues are $-r < 0$, $\lambda k-d $,
$-d-\alpha < 0$, and $-\eta<0$, so $E_1$ is stable only if $\lambda k-d<0$.
This stability is well characterized through the basic reproduction number, denoted by $R_0$, 
which is one of the most important quantities in epidemiology. 

\begin{theorem}[Basic reproduction number $R_0$]
\label{thm:R0}
The basic reproduction number of system (\ref{delay1}) 
is given by 
\[
R_0=\frac{m\lambda K(\eta+\omega)}{d(\eta+\omega)+h\omega}.
\]
\end{theorem}

\begin{proof}
We use the next-generation matrix method.
Let $Y=(S,I)^T$. Then, from system \eqref{delay1}, we can write
\[
Y'=F(Y)-V(Y),
\]
where
\[
V(Y)=\left[
\begin{array}{c}
\	m\lambda XS\\
0
\end{array}
\right] 
\quad \mbox{and}\quad
F(Y)=
\left[
\begin{array}{c}
\frac{hAS}{1+A}+dS\\
-\frac{hAS}{1+A}+(d+\alpha)I
\end{array}
\right].
\]
The Jacobian matrix of $F(Y)$ and $V(Y)$, at the pest free equilibrium $E_1$, 
are given respectively as
\[
DF(E_1)=\left[
\begin{array}{cc}
m\lambda K & 0\\
0&  0  \\
\end{array}
\right]
\quad \text{ and } \quad
DV(E_1)=\left[
\begin{array}{cc}
\frac{h\omega}{\eta+\omega}+d& 0\\
-\frac{h\omega}{\eta+\omega}&  d+\alpha \\
\end{array}
\right]
\]
with
\[
(DV(E_1))^{-1}=\left[
\begin{array}{cc}
\frac{\eta+\omega}{h\omega+d(\eta+\omega)} & 0\\
\frac{h\omega}{(h\omega+d(\eta+\omega))(d+\alpha)}&  \frac{1}{d+\alpha}  \\
\end{array}
\right].
\]
The reproduction number $R_0$ is given by the spectral radius 
of $DF(E_1)(DV(E_1))^{-1}$, that is,
\[
R_0=\sigma(DF(E_1)(DV(E_1))^{-1})
=\sigma\left(
\left[
\begin{array}{cc}
\frac{m\lambda K(\eta+\omega)}{h\omega+d(\eta+\omega)}& 0\\
0&  0 \\
\end{array}
\right]\right)
=\frac{m\lambda K(\eta+\omega)}{d(\eta+\omega)+h\omega}.
\]
The proof is complete.
\end{proof}

\begin{remark}
\label{Rem1}
It is fundamental to note that $R_0$ does not depend on $a$. For this reason,
irrespective of how capably farmers turn aware of disease due to inspection
of infected plants, this by itself is not sufficient to result in the eradication
of infection. Because $R_0$ is monotonically decreasing with increasing of $\omega$,
this suggests that eradication of plant disease, as represented by a stable pest-free
steady-state $E_1$, is only possible if $R_0<1$. One existing means to achieve this
consists to increase the rate $\omega$ of global awareness.
\end{remark}

We finish this section characterizing the locally stability
of the the pest-free equilibrium $E_1$ 
in terms of the basic reproduction number $R_0$
given by Theorem~\ref{thm:R0}.

\begin{theorem}[Local stability of the pest-free equilibrium]
For $R_0<1$, the pest-free equilibrium $E_1$ is locally stable.
For $R_0>1$, $E_1$ is unstable and the coexistence equilibria $E^*$ exists.
A transcritical bifurcation occurs at $R_0=1$.
\end{theorem}

\begin{proof}
The characteristic equation at the pest free equilibrium $E_1$ is given by
\[
(\xi+r)(\xi+\eta)\left(\xi-m\lambda K+\frac{h\omega}{\eta+\omega} +d\right)
\left(\xi +d+ \alpha\right) = 0.
\]
Thus, the eigenvalues are $-r$ ,$-\eta$,$-d-\alpha$, 
and $m\lambda K-\frac{h\omega}{\eta+\omega}-d$.
Therefore, $E_1$ is locally stable if 
$m\lambda K-\frac{h\omega}{\eta+\omega}-d<0$,
from which $m\lambda K<h\omega+d(\eta+\omega)$,
that is,
$$
\frac{m\lambda K}{h\omega+d(\eta+\omega)}<1.
$$
We conclude that $R_0<1$, as intended.
\end{proof}
	

\section{Stability and Hopf Bifurcation in Absence of Delays}
\label{sec:5a}

In this section we consider $\tau_1=\tau_2=0$, investigating the direction 
and stability of the bifurcating periodic solution. More precisely, 
we focus our analysis on the consumption rate $\lambda$ of pest, 
which is one of the biologically most important parameters of the model. 

\begin{theorem}[Local stability of the coexistence equilibrium---undelayed case]
\label{thm:7}
Let us consider system \eqref{delay1} with $\tau_1=\tau_2=0$
and let $E^* = (X^*,S^*,I^*,A^*)$ be its coexistence equilibrium. 
Then $E^*$ is locally asymptotically stable if, and only if,
\begin{equation}
\label{ChE}
\sigma_3>0 \quad \text{ and } \quad
\sigma_1\sigma_2-\sigma_3>0,
\end{equation}
where
\begin{equation}
\label{eq:sigma:values}
\begin{split}
\sigma_1&=\frac{rX^*}{K}+\alpha+d +\eta,\\
\sigma_2&=\frac{\eta r^2X^*}{K}+\frac{ahS^*}{(1+A^*)^2}+m\lambda^2X^*S^*,\\
\sigma_3&=-\frac{r^2 ahS^*X^*}{K(1+A^*)^2}+m\eta\lambda^2X^*S^*,
\end{split}
\end{equation}
and
$$
\sigma_1 \sigma_2 -\sigma_3
=\frac{r^2X^*\eta}{K}\left(\frac{r^2X^*}{K}+\eta\right)
+ \frac{a\eta hS^*}{(1+A^*)^2}-\frac{m\lambda^2{X^*}^2S^*r^2}{K}.
$$
\end{theorem}

\begin{proof}
The Jacobian of system \eqref{delay1} without
delays at the coexistence state $E^*$ is
\[
J(X^*,S^*,I^*,A^*)
=\left[
\begin{array}{cccc}
r-\frac{2rX^*}{K}-\lambda S^*  & -\lambda X^* & 0& 0\\
m\lambda S^*& m\lambda X^*-\frac{hA^*}{1+A^*}-d & 0&\frac{hS^*}{(1+A^*)^2}\\
0& \frac{h A^*}{1+A^*} &-(d+\alpha)&-\frac{hS^*}{(1+A^*)^2} \\
0&a&0&-\eta
\end{array}
\right].
\]
The characteristic equation in $\xi$ for the Jacobian matrix $J(E^*)$ is given by
\[
|\xi I-J(E^*)|=0,
\]
that is,
\[
\left|
\begin{array}{cccc}
\xi-r+\frac{2rX^*}{K}+\lambda S^*  & \lambda X^* & 0& 0\\
-m\lambda S^*& \xi-m\lambda X^*+\frac{hA^*}{1+A^*}+d & 0&-\frac{hS^*}{(1+A^*)^2}\\
0&- \frac{h A^*}{1+A^*} &\xi+(d+\alpha)&\frac{hS^*}{(1+A^*)^2} \\
0&-a&0&\xi+\eta
\end{array}
\right|=0,
\]
which gives
\begin{eqnarray}
\label{char_eqn}
(\xi+d+\alpha)(\xi^3+\sigma_1\xi^2 +\sigma_2\xi + \sigma_3) = 0
\end{eqnarray}
with $\sigma_1$, $\sigma_2$ and $\sigma_3$ as in \eqref{eq:sigma:values}.
By the Routh--Hurwitz theorem, it follows that the coexistence equilibrium 
$E^*$ is locally asymptotically stable if, and only if, $\sigma_1>0$, $\sigma_2>0$, 
$\sigma_3>0$, and $\sigma_1\sigma_2-\sigma_3>0$. Due to the positivity of parameters, 
$\sigma_1$ and $\sigma_2$ are always positive, so the stability condition 
can be given by \eqref{ChE}.
\end{proof}
		
Hopf bifurcation of the endemic steady state can occur if the characteristic equation 
\eqref{char_eqn} has a pair of purely imaginary eigenvalues for some 
$\lambda=\lambda^*\in(0,\infty)$ and, additionally, all other eigenvalues 
have negative real parts \cite{CMT}. 
The characteristic equation \eqref{char_eqn} 
has one negative root, namely, $-(d+\alpha)$. Hopf bifurcation 
occurs according to the following theorem.

\begin{theorem}[Hopf bifurcation---undelayed case]
Let $\sigma_i$, $i = 1, 2, 3$, be given by \eqref{eq:sigma:values}.
System \eqref{delay1} with $\tau_1=\tau_2=0$, at the endemic equilibrium point $E^*$, 
undergoes Hopf bifurcation at $\lambda = \lambda^*$ in the following domain:
\[
\Lambda_{HB}
=\left\{\lambda \in\mathbb{R}^+:\sigma_1(\lambda^*)\sigma_2(\lambda^*)
-\sigma_3(\lambda^*)=0 \ \text{ with }\ \sigma_2>0,\,
\dot{\sigma}_3-(\dot{\sigma_1}\sigma_2+\sigma_1\dot{\sigma_2})\ne0\right\},
\]
where we use the 'dot' to denote differentiation with respect to $\lambda$.
\end{theorem}

\begin{proof}
Using the condition $\sigma_1\sigma_2-\sigma_3=0$, 
the characteristic equation \eqref{char_eqn} becomes
\[
(\xi^2+\sigma_2)(\xi+\sigma_1)=0,
\]
which has three roots: $\xi_1=i\sqrt{\sigma_2}$, $\xi_2=-i\sqrt{\sigma_2}$, 
and $\xi_3=-\sigma_1$. Therefore, a pair of purely 
imaginary eigenvalues exists for $\sigma_1\sigma_2-\sigma_3=0$. 
Now, we verify the transversality condition. Differentiating the 
characteristic equation \eqref{char_eqn} with respect to $\lambda$, 
we have
\begin{eqnarray*}
\frac{d\xi}{d\lambda}
&=&{\frac{\xi^2\dot{\sigma_1}+\xi\dot{\sigma_2}
+\dot{\sigma_3}}{3\xi^2+2\xi\sigma_1+\sigma_2}}|_{\xi=i\sqrt{\sigma_2}}\\
&=&\frac{\dot{\sigma_3}-(\dot{\sigma_1}\sigma_2+\sigma_1
\dot{\sigma_2})}{2(\sigma_1^2+\sigma_2)}+i\left[\frac{\sqrt{\sigma_2}
(\sigma_1\dot{\sigma_3}+\sigma_2\dot{\sigma_2}
-\sigma_1\dot{\sigma_1}\sigma_2)}{2\sigma_2(\sigma_1^2+\sigma_2)} \right].
\end{eqnarray*}
Therefore,
\[
\left.\frac{dRe \xi}{d\lambda}\right|_{\lambda=\lambda^*}=\frac{\dot{\sigma_3}
-(\dot{\sigma_1}\sigma_2+\sigma_1 \dot{\sigma_2})}{2(\sigma_1^2
+\sigma_2)}\ne 0
\iff 
\dot{\sigma_3}-(\dot{\sigma_1}\sigma_2
+\sigma_1 \dot{\sigma_2}\ne 0
\]
and Hopf bifurcation occurs at $\lambda=\lambda^*$.
\end{proof}

\begin{remark}
When the parameter $\lambda$ crosses the critical value $\lambda^*$, 
a limit cycle of system \eqref{delay1} occurs around $E^*$.
\end{remark}


\section{Stability and Hopf Bifurcation of the Delayed System}
\label{sec4.2}

In this section, we investigate stability and Hopf bifurcation 
for the delayed system \eqref{delay1}. Without loss of generality, 
it is assumed that $E^* = (X^*, S^*, I^*, A^*)$ 
is the interior equilibrium point of the system with delays.
In the sequel, we define $\tau$ as the sum of the two delays 
of the system, that is,
$$
\tau := \tau_1+\tau_2.
$$ 

\begin{theorem}[Local stability of the coexistence equilibrium---delayed case]
\label{prop4}
Let $a_i$, $i=1,\ldots,5$, be given by \eqref{eq:val:a:i}. Define
$S_1:=a_1^2-2a_2$, $S_2:=a_2^2-2a_1a_5-a_4^2$, and $S_3:=a_5^2-a_3^2$.
If the conditions
\begin{enumerate}
\item[(i)] $a_1>0$, $a_3+a_5>0$, and $a_1(a_2+a_4)-(a_3+a_5)>0$;
\item[(ii)] $S_1\geq 0$, $S_3\geq 0$, and $S_2>0$;
\end{enumerate}
are satisfied, then the infected steady state $E^*$
is locally asymptotically stable for all $\tau \geq 0$.
\end{theorem}

\begin{proof}
For the stability of the endemic equilibrium, the distribution of the roots 
of \eqref{3} needs to be analyzed. 
The characteristic equation \eqref{3} becomes
\begin{equation}
\label{char3}
\psi(\xi,\tau)=\xi^3+a_1\xi^2+a_2\xi+a_5+[a_3+ a_4\xi]e^{-\xi\tau}=0.
\end{equation}
Equation \eqref{char3} is a transcendental equation in $\xi$ with infinitely many roots. 
It is known that the coexistence equilibrium $E^*$ is locally stable (or unstable) 
if all the roots of the corresponding characteristic equation have negative real parts 
(or have positive real parts). Suppose $\xi = i w( \tau)$ is a root 
of equation (\ref{char3}). Then,
\begin{equation}
\label{2.10}
-iw^3-a_1w^2+ia_2w+(a_4iw+a_3)(\cos{w\tau}-i\sin{w\tau})+a_5=0.
\end{equation} 	
Separating the real and imaginary parts, we obtain the following equations:
\begin{equation}	
\label{2.11}
w^3-a_2w=a_4w\cos{w\tau}-a_3\sin{w\tau},
\end{equation}
\begin{equation}	
\label{2.12}
a_1w^2-a_5=a_3\cos{w\tau}+a_4w\sin{w\tau}.
\end{equation}
Squaring and adding the real and imaginary parts, we get
\begin{equation}	
\label{2.13}
w^6+(a_1^2-2a_2)w^4+(a_2^2-2a_1a_5-a_4^2)w^2+(a_5^2-a_3^2)=0.
\end{equation}
Let $\theta=w^2$. Then, equation (\ref{2.13}) becomes
\begin{equation}	
\label{2.14}
H(\theta)=\theta^3+S_1\theta^2+S_2\theta+S_3=0.
\end{equation}
Now, if the following conditions
\begin{equation*}
S_1=a_1^2-2a_2\geq 0,
\quad S_2=a_2^2-2a_1a_5-a_4^2>0,
\quad S_3=a_5^2-a_3^2\geq 0
\end{equation*}
are satisfied, then equation (\ref{2.14}) has no positive real roots.
We get from (\ref{2.14}) that
\begin{equation} 	
\label{2.15}
\frac{dH(\theta)}{d\theta}=3\theta^2+2S_1\theta+S_2=0,
\end{equation}
which has two roots
\begin{equation}	
\label{2.16}
\theta_{1}=\frac{-S_1+\sqrt{S_1^2-3S_2}}{3},
\quad \theta_{2}=\frac{-S_1-\sqrt{S_1^2-3S_2}}{3}.
\end{equation}
Since we assume $S_2>0$, we have $\sqrt{S_1^2-3S_2}<S_1$
and hence neither $\theta_1$ nor $\theta_2$ is positive.
Thus, equation (\ref{2.15}) does not have any positive roots.
Since $F(0)=S_3\geq 0$, we conclude that equation (\ref{2.14}) 
has no positive roots. 
\end{proof}
	
\begin{remark}
The parameters in Table~\ref{tab3.1}
satisfy the conditions of Theorem~\ref{prop4}. Thus,
the steady state $E^*$ of the delayed model \eqref{delay1}
is asymptotically stable for all $\tau>0$, i.e.,
the stability of the system at the coexistence equilibrium 
is delay independent.
\end{remark}
	
A Hopf-bifurcating periodic solution appears
when purely imaginary roots exist. We shall now
check the possible occurrence of Hopf-bifurcation.
	
If $S_3<0$, then there exists a positive root $\theta_0$ for which the characteristic
equation \eqref{3} has pair of purely imaginary roots $\pm iw_0$.
Equation (\ref{2.14}) satisfies $H(0)<0$ and
$\lim_{\theta\rightarrow\infty} H(\theta)=\infty$. Thus, equation (\ref{2.14})
has at least one positive root, $\theta_0$. Again, if $S_2<0$, then
$\sqrt{S_1^2-3S_2}>S_1$ and hence $\theta_1>0$. This implies that equation
(\ref{2.13}) possesses purely imaginary roots $\pm iw_0$.
For $w(\tau_0)=w_0$, equations (\ref{2.11}) and (\ref{2.12}) give
\begin{equation}	
\label{2.17}
\tau_{n}=\frac{1}{w_0}{\arccos
\left[\frac{(a_4{w_0}^4)-(a_2a_4-a_1a_3){w_0}^2-a_3a_5}
{a_4^2w_0^2+a_3^2}\right]}  +\frac{2n\pi}{w_0},
\quad n=0, 1, 2, \ldots,
\end{equation}
where 
\begin{equation}
\label{eq:w0}
w_{0}^2=\frac{-S_1+\sqrt{{S_1}^2-3S_2}}{3}.
\end{equation}
Therefore, the following result holds.
	
\begin{theorem}[Hopf bifurcation---delayed case]
\label{thm:H:b:dc}
Let $a_i$, $i=1,\ldots,5$, and $S_j$, $j=1,\ldots,3$,
be as in Theorem~\ref{prop4}. Suppose that the interior 
equilibrium point $E^*$ exists and $a_1>0$, $a_3+a_5>0$, 
and $a_1(a_2+a_4)-(a_3+a_5)>0$. Define
\begin{equation}
\label{2.18}
\tau_{0}:=\frac{1}{w_0}{\arccos
\left[\frac{(a_4{w_0}^4)-(a_2a_4-a_1a_3){w_0}^2-a_3a_5}
{a_4^2w_0^2+a_3^2}\right]},
\end{equation}
where $w_0$ is given by \eqref{eq:w0}.
If either $S_3<0$ or $S_3\geq0$ and $S_2<0$, 
then $E^*$ is asymptotically stable when $\tau<\tau_0$
and unstable when $\tau>\tau_0$.
When $\tau=\tau_0$, a Hopf bifurcation occurs: a family of periodic
solutions bifurcates at $E^*$ as $\tau$ passes through the critical
value $\tau_0$, provided the transversality condition
\begin{equation}	
P(w_0)R(w_0)-Q(w_0)S(w_0) > 0
\end{equation}
is satisfied with
\begin{eqnarray*}
P(w)&=&-3w^2+a_1+a_4\cos w\tau-\tau(a_3\cos w\tau+a_4w\sin w\tau),\\
Q(w)&=&-ca_1w+a_4\sin w\tau-\tau(a_3\sin w\tau-a_4w\cos w\tau),\\
R(w)&=&a_3w\sin w\tau-a_4w^2\cos w\tau,\\
S(w)&=&a_3w\cos w\tau+a_4w^2\sin w\tau.
\end{eqnarray*}
\end{theorem}
	
\begin{proof}
Differentiating (\ref{2.11}) and (\ref{2.12}) with respect to $\tau$, we have
\begin{equation}	
\frac{d}{d\tau}[Re \{\lambda(\tau)\}]_{\tau=\tau_0,w=w_0}
=\left.\frac{P(w)R(w)-Q(w)S(w)}{P^2(w)+Q^2(w)}\right|_{\tau=\tau_0,w=w_0}.
\end{equation}
Therefore, $\frac{d}{d\tau}[Re \{\lambda(\tau)\}]_{\tau=\tau_0,w=w_0}>0$
if $P(w_0)R(w_0)-Q(w_0)S(w_0) > 0$. Thus, the transversality condition
is satisfied and hence a Hopf bifurcation occurs at $\tau=\tau_0$ \cite{Buttler}.
\end{proof}
	
The full characterization of the direction and stability of the bifurcating 
periodic solution for the delayed system \eqref{delay1} is cumbersome
but can be done following the approach of \cite{IACM,KH18}. In Section~\ref{sec:6} 
we illustrate the nature of Hopf bifurcation through numerical simulations. 


\section{Optimal Control}
\label{sec:5}
	
Optimal control with delays is a subject under strong current research
\cite{MR3259239,MR2970905,MyID:405}, in particular with respect to applications
in biological control systems \cite{MyID:400,MR3562914,MR4086474}.
Our main aim here is to find the optimal profile $u_1(t)$ and $u_2(t)$ 
to minimize the cost of biopesticides used in pest control.
We introduce the parameters $u_1(t)$ and $u_2(t)$ as control parameters: 
$u_1(t)$ is an admissible control representing the efficiency of pesticide 
being used and $u_2(t)$ characterizes the cost of the awareness campaign.
The control induced state system is given as follows:
\begin{equation}
\label{state1}
\begin{cases}
\frac{dX}{dt} = r X\left[1 - \frac{X}{K}\right] - \lambda XS,\\
\frac{dS}{dt} = m\lambda XS - (1-u_1)\frac{hA(t-\tau_2)S}{1+A(t-\tau_2)} -dS,\\
\frac{dI}{dt} = (1-u_1)\frac{hA(t-\tau_2)S}{1+A(t-\tau_2)} - (d+\alpha)I,\\
\frac{dA}{dt} = u_2\omega+aS(t-\tau_1) - \eta A,
\end{cases}
\end{equation}
subject to initial conditions \eqref{eq:ic}.
We want to minimize the pest and also the cost of pest management.
With this in mind, we define the cost functional for the minimization problem as
\begin{equation}
\label{1.6}
J(u_1,u_2)=\int^{t_f}_{0}[B_1u_1^2(t)+B_2u_2^2(t)+CS^2(t)]dt
\end{equation}
subject to the control system (\ref{state1}). The parameters $B_i$, $i=1,2$, 
represent a weight constant on the benefit of the cost of production, while $C$
is the penalty multiplier. Our aim is to find the optimal control 
$u^*(\cdot)=(u_1(\cdot), u_2(\cdot))$ such that
\begin{equation}
\label{eq:J:min}
J(u^*)= \min ~(J(u):u\in U_1\times U_2),
\end{equation}
where
\begin{equation}
\label{eq:U}
U_1\times U_2=\left\{u : ~u~\mbox{is measurable and}~
0\leq u_i(t)\leq 1,~t\in[0, t_f]\right\}.
\end{equation}
The Pontryagin Minimum Principle with delays in the state system
can be found in \cite{Gollmann,DT1} and references therein,
providing necessary optimality conditions for our delayed optimal control problem.
Roughly speaking, Pontryagin's Principle reduces problem \eqref{state1}--\eqref{eq:U}
to a problem of minimizing the Hamiltonian $H$ given by
\begin{equation}
\label{hami}
\begin{split}
H=B_1u_1^2&+B_2u_2^2+CS^2+\xi_1\left\{rX(1-{X}{K}^{-1}) - \lambda XS \right\}\\
&+\xi_2\left\{ m\lambda XS - (1-u_1)\frac{hA(t-\tau_2) S}{1+A(t-\tau_2)} -dS\right\}\\
&+\xi_3\left\{ (1-u_1)\frac{hA(t-\tau_2) S}{1+A(t-\tau_2)} - (d+\alpha)I\right\}
+\xi_4\left\{u_2\omega+aS(t-\tau_1) - \eta A \right\}.
\end{split}
\end{equation}
Precisely, the following theorem characterizes the solution of our delayed
optimal control problem.

\begin{theorem}[Necessary optimality conditions for the delayed optimal control problem]
\label{thm:PMP:ourOCP}
Let $u^*$ be solution to the optimal control problem \eqref{state1}--\eqref{eq:U}
and $\left(X^*, S^*, I^*, A^*\right)$ be the corresponding (optimal) state.
Then there exist multipliers $\xi_1$, $\xi_2$, $\xi_3$ and $\xi_4$
that satisfy the adjoint system
\begin{equation}
\label{costate2}
\begin{cases}
\frac{d\xi_1}{dt}
= -\xi_1[r(1- {2X^*}{K}^{-1}) -\lambda S^*]-\xi_2m\lambda S^*,\\
\frac{d\xi_2}{dt}=
-2CS^*+\lambda X^*\xi_1+\left((1-u_1)\frac{hA^*(t-\tau)}{1
+A^*(t-\tau)}-m\lambda X^*+d\right)\xi_2\\
\quad -\left((1-u_1)\frac{hA^*(t-\tau)}{1+A^*(t-\tau)}\right)\xi_3
+\chi_{[0,t_f-\tau]}(t)a\xi_4(t+\tau),\\
\frac{d\xi_3}{dt}=(d+\alpha)\xi_3,\\
\frac{d\xi_4}{dt}
=\xi_4\eta+\chi_{[0,t_f-\tau]}(t)(1-u_1)\frac{hS^*(t+\tau)}{(1
+A^*)^2}\left[\xi_2(t+\tau)-\xi_3(t+\tau)\right],
\end{cases}
\end{equation}
where $\tau=\max\{\tau_1,\tau_2\}$, subject
to the transversality conditions $\xi_i(t_f)=0$, $i=1, \ldots, 4$.
Moreover, the optimal control is given by
\begin{equation}
\label{1.13}
\begin{gathered}
u_1^*(t)=\max\left\{0,~\min\left\{1,~\frac{hA^*(t-\tau)S^*(t)
(\xi_3(t)-\xi_2(t))}{2B_1 (1+A^*(t-\tau))}\right\}\right\},\\
u_2^*(t)=\max\left\{0,~\min\left\{1,~-\frac{\xi_4\omega}{2B_2}\right\}\right\}.
\end{gathered}
\end{equation}
\end{theorem}
	
\begin{proof}
The result is a direct consequence of the Pontryagin Minimum Principle with delays,
which asserts that the solution to the delayed optimal control problem
satisfies the adjoint equations
\begin{equation}
\label{1.9}
\begin{split}
\frac{d\xi_1}{dt}(t)&=-\frac{\partial H}{\partial X}(t),
\quad \frac{d\xi_2}{dt}(t)=-\frac{\partial H}{\partial S}(t)
-\chi_{[0,~t_f-\tau]}(t)\frac{\partial H}{\partial S}(t+\tau),\\
\frac{d\xi_3}{dt}(t)&=-\frac{\partial H}{\partial I}(t),
\quad \frac{d\xi_4}{dt}(t)=-\frac{\partial H}{\partial A}(t)
-\chi_{[0,~t_f-\tau]}(t)\frac{\partial H}{\partial A}(t+\tau),
\end{split}
\end{equation}
and the transversality conditions $\xi_i(t_f)=0$, $i=1,~2,~3,~4$. Moreover,
according to the Pontryagin Minimum Principle, when the optimal
control $u^*$ takes values on $(0,1)$, then
$\frac{\partial H}{\partial u_i}(t)=0,i=1,2$.
Now, taking partial differentiation of \eqref{hami} 
with respect to $u_1$ and $u_2$, we get
\begin{equation}
u_1^*(t)=\frac{(\xi_3(t)-\xi_2(t))hA^*(t-\tau)S^*(t)}{2B_1(1+A^*(t-\tau))},
\quad u_2^*(t)=\frac{-\xi_4(t)\omega}{2B_2}.
\end{equation}
Using the boundedness of the controls, i.e., the fact that admissible
controls take values such that $0 \leq u_i(t) \leq 1, i=1,2$, it follows
from the minimality condition of Pontryagin's Principle that
\begin{align*}
&u_1^*(t)=\max\left\{0,~\min\left\{1,~\frac{hA^*(t-\tau)S^*(t)
(\xi_3(t)-\xi_2(t))}{2B_1 (1+A^*(t-\tau))}\right\}\right\},\\
&u_2^*(t)=\max\left\{0,~\min\left\{1,~-\frac{\xi_4(t)\omega}{2B_2}\right\}\right\}.
\end{align*}
The proof is complete.
\end{proof}
	
State system (\ref{state1}) subject to \eqref{eq:ic}
is an initial value problem while the co-state system (\ref{costate2})
subject to the transversality conditions $\xi_i(t_f)=0$, $i=1, \ldots, 4$,
is a terminal value problem. For this reason, one solves numerically
the boundary value problem given by Theorem~\ref{thm:PMP:ourOCP}
through forward iteration in the state system
while the co-state system (\ref{costate2}) is integrated
through backward iteration \cite{MR4091761}.
Numerical results are given in Section~\ref{sec:ns:oc}.
	
	
\section{Numerical Simulations}
\label{sec:6}
	
In this section, we provide numerical simulations obtained
from the application of our analytical results, as given
in previous sections. To illustrate the behavior of
model (\ref{delay1}), we did numerical simulations
with the set of parameter values in Table~\ref{tab3.1}.
Some of these values are estimated from \cite{Bhatta2,mma,ns}.
We begin by simulating the system without delays, then with delays.
	
	
\subsection{Numerical simulations of the stability of equilibria}
	
The parameters of model \eqref{delay1} used in the numerical
simulations are given in Table~\ref{tab3.1}. The initial values
are chosen as
\begin{equation}
\label{initial}
(X(\phi), S(\phi), I(\phi), A(\phi)) = (0.02, 0.12, 0.02, 0.35),
\quad \phi\in [-\tau, 0],
\end{equation}
where $\tau=\max\{\tau_1,\tau_2\}$.
\begin{figure}[H]
\centering
\includegraphics[scale=0.70]{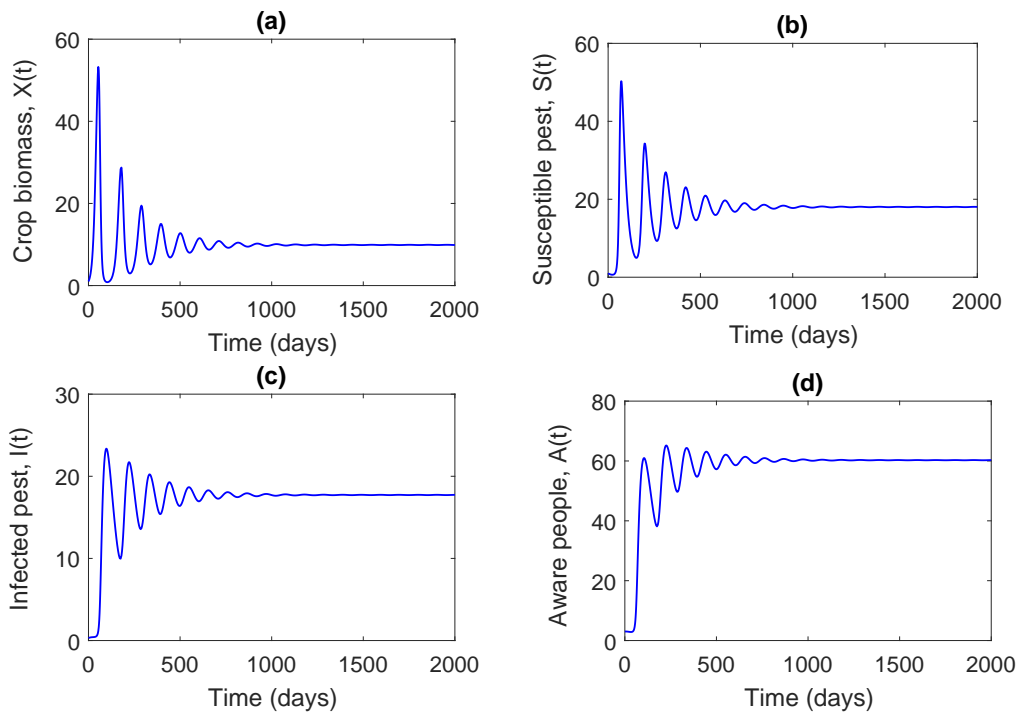}
\caption{Stability behavior of the endemic state
for	the non-delayed population system \eqref{delay1}, i.e.,
when $\tau_1=0$ and $\tau_2=0$.}
\label{fig:1}
\end{figure}
When we choose $\tau_1=0$ and $\tau_2=0$, the behavioral patterns
of the densities of susceptible pest, infected pest, crop biomass,
and aware population are presented taking the parameters from Table~\ref{tab3.1}.
All the system populations oscillate initially and finally become asymptotically stable
and converge to the endemic state value (see Figure~\ref{fig:1}). Here,
it is to be noted that all conditions of Theorem~\ref{thm:7}
are satisfied and thus the coexistence equilibrium $E^*$ is stable.

Figure~\ref{fig:2} contains the contour plot of the basic reproduction number $R_0$, 
given by Theorem~\ref{thm:R0}, as a function of the awareness active rate $h$ and global 
awareness rate $\omega$. The region where $E^*$ exists and $E_1$ is stable
are identified. We see that if both $h$ and $\omega$ are large, then the coexistence
equilibria $E^*$ will not be feasible but the pest-free equilibrium $E_1$ is feasible 
and stable as $R_0<1$ there. Consequently, a transcritical bifurcation occurs at $R_0=1$. 
This is also clear from Figure~\ref{fig:3}.
\begin{figure}[H]
\centering
\includegraphics[scale=0.70]{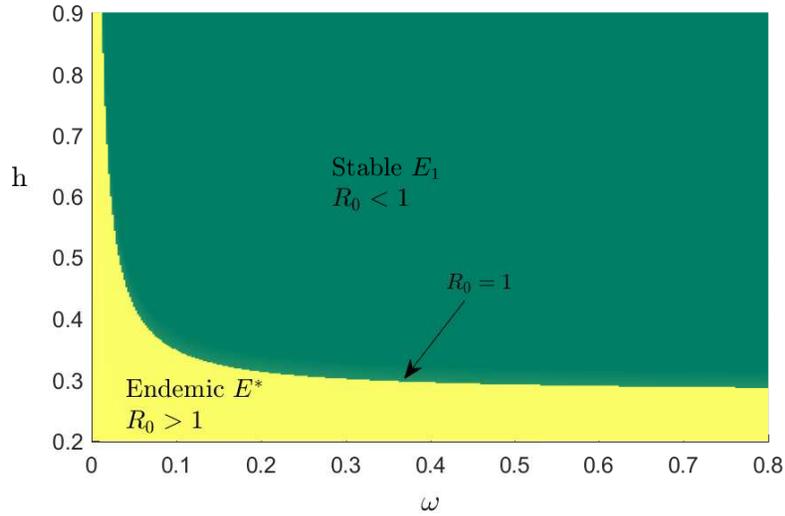}
\caption{Plot of the basic reproduction number $R_0$ as a function of
$h$ and $\omega$. All other parameters are as in Table~\ref{tab3.1}.}
\label{fig:2}
\end{figure}
\begin{figure}[H]
\centering
\includegraphics[scale=0.7]{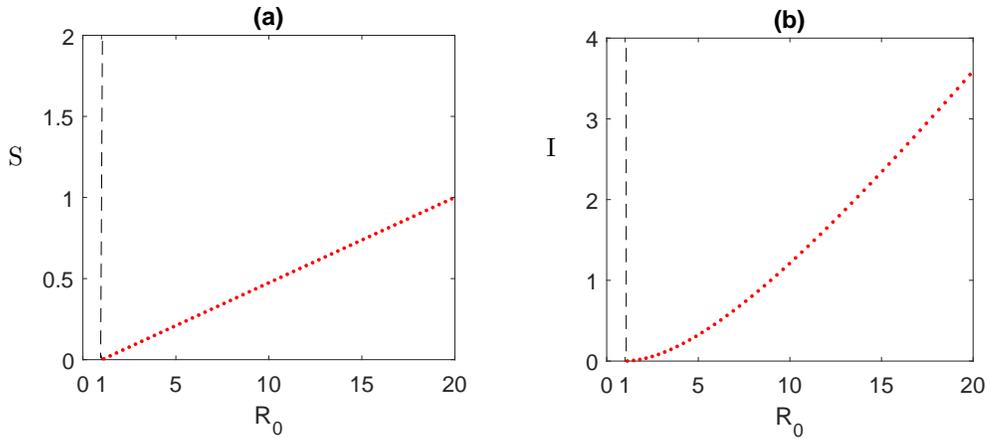}
\caption{Forward bifurcation. Plot of the basic reproduction number $R_0$ 
as a function of $S$ and $I$. All other parameters are as in Table~\ref{tab3.1}. 
Coexisting equilibrium $E^*$ exists for $R_0>1$ and for $R_0<1$ it is not 
feasible but pest-free equilibrium $E_1$ is stable.}
\label{fig:3}
\end{figure}
Figure~\ref{fig:4} displays the effect of local awareness rate $a$
on the steady state values of the model populations (without delay). As expected, 
we see that awareness of population is increased as the value of $a$ increases 
and the uninfected pest population decreases.
\begin{figure}[H]
\centering
\includegraphics[scale=0.77]{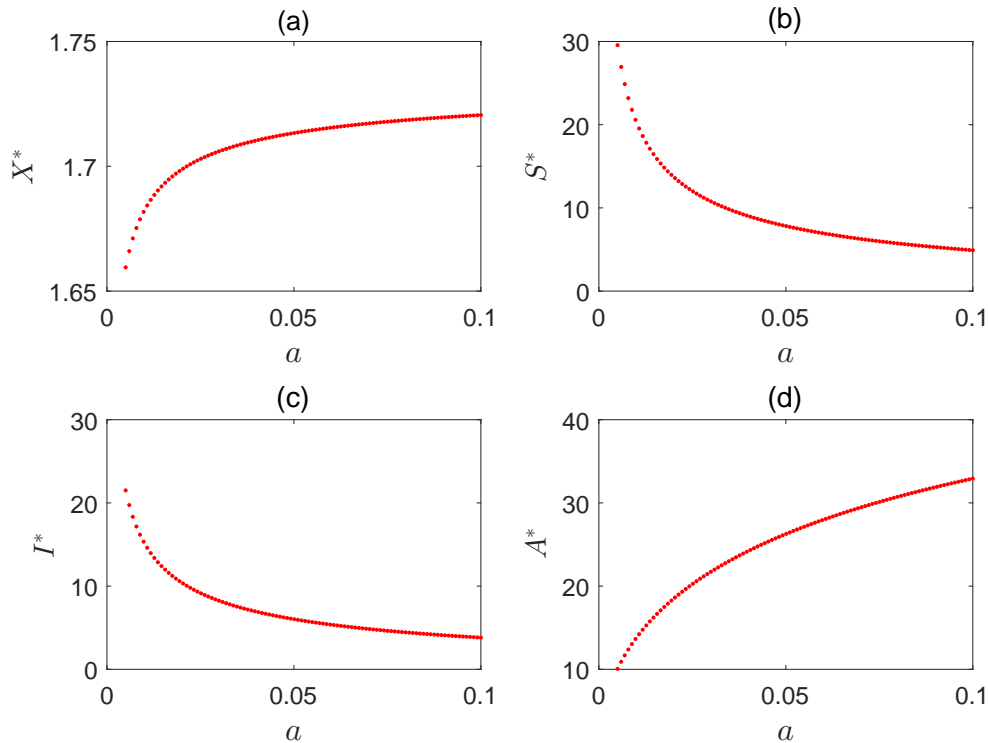}
\caption{Effect of local awareness, $a$, on the nondelayed system \eqref{delay1} 
with values of the other parameters given in Table~\ref{tab3.1}.}\label{fig:4}
\end{figure}
	
	
\subsection{Numerical simulations with different time delays}
	
In this part, we examine the influence of the time delays on the number 
of infected pest by performing some numerical simulations. 

When the time delay $\tau_2 = 0$ is fixed and $\tau_1$ 
takes values $\tau_1 = 6$ and $\tau_1 = 16$, 
respectively, we observe that oscillation increases as 
$\tau_1$ increases and periodic oscillation is seen at 
$\tau_1=16$ days (see Figure~\ref{fig:5}). 
Stability switch occurs through Hopf bifurcation.

In Figure~\ref{fig:6}, the time delay $\tau_1 = 0$ is fixed and $\tau_2$ 
takes value $\tau_2 = 16$. We see a similar result as the one of Figure~\ref{fig:5}. 

For any combination of $\tau_1$ and $\tau_2$, 
when $\tau_1+\tau_2=16$, periodic oscillation (Hopf bifurcation) 
can be seen. Figure~\ref{fig:7} confirms that the Hopf bifurcation 
is of supercritical type.

Figures~\ref{fig:5}, \ref{fig:6} and \ref{fig:7} indicate that the populations always
oscillate initially and ultimately approach towards their equilibrium values when
$\tau_1<\tau_0=16$. This indicates that the number of pest will be high and sometimes low.
If the value of delay, $\tau_i$, exceeds its critical value, $\tau_0=16$ days approximately,
the population becomes periodic. In this situation, it is very difficult to make the prediction
about the size of the epidemic. Hence, these figures indicate that the endemic equilibrium $E^*$
of model system \eqref{delay1} with delay is stable for $\tau < \tau_0$ and unstable
for $\tau > \tau_0$. For $\tau > \tau_0$, bifurcating periodic solutions are observed.
This is in agreement with Theorem~\ref{thm:H:b:dc}.

In Figure~\ref{fig:8}, the region of stability in the $\lambda$--$\tau_2$ parameter plane is shown. 
The critical value of $\tau_2$, for which Hopf bifurcation occurs, depends on the consumption 
rate $\lambda$. The critical value of $\tau_2$ decreases as $\lambda$ increases.

\begin{figure}[H]
\centering
\includegraphics[scale=0.7]{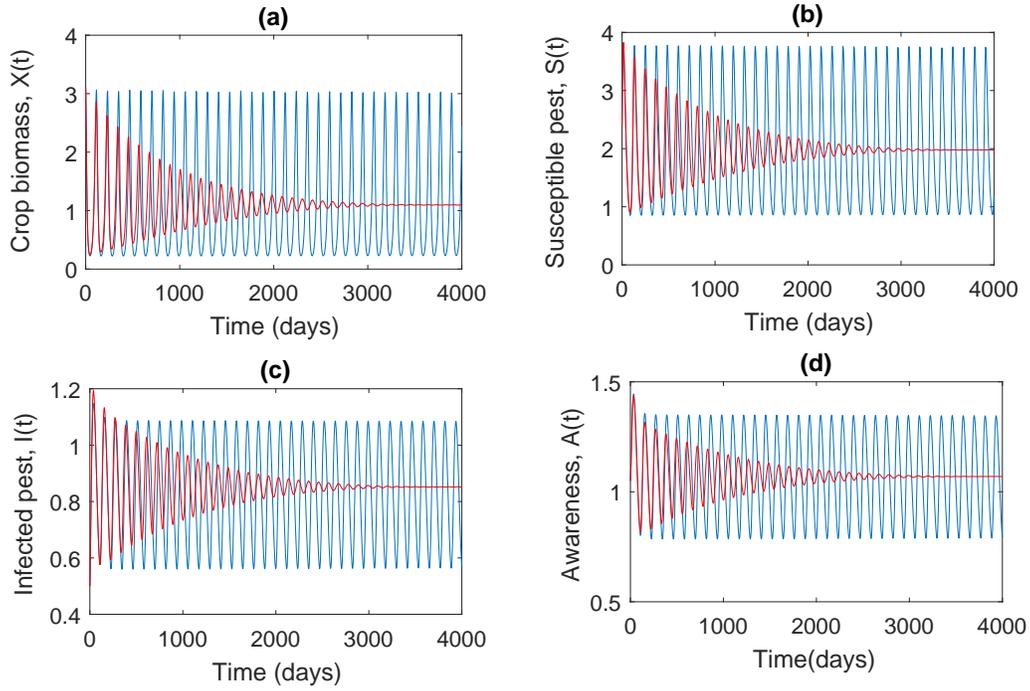}
\caption{The behavior of the model population for different values of $\tau_1$ 
and fixed $\tau_2=0$. Red line represents $\tau_1=6$ and blue line represents 
$\tau_1=16$. Other parameters are as in Figure~\ref{fig:1}.}\label{fig:5}
\end{figure}
\begin{figure}[H]
\centering
\includegraphics[scale=0.7]{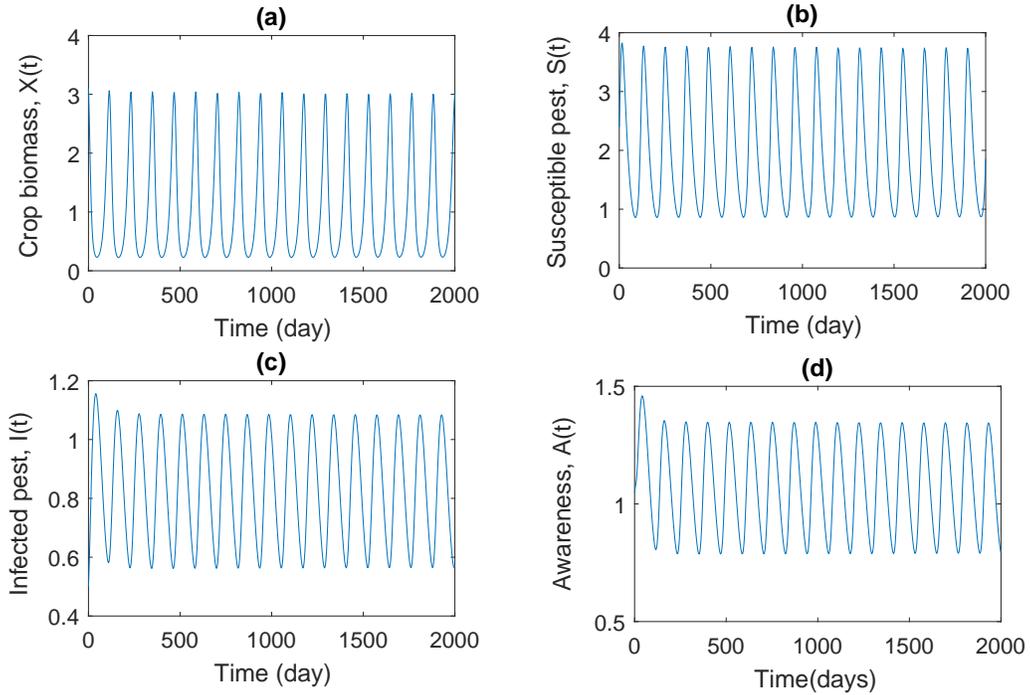}
\caption{The behavior of the model population for $\tau_2=16$ 
and fixed $\tau_1=0$.}\label{fig:6}
\end{figure}
\begin{figure}[H]
\centering
\includegraphics[scale=0.7]{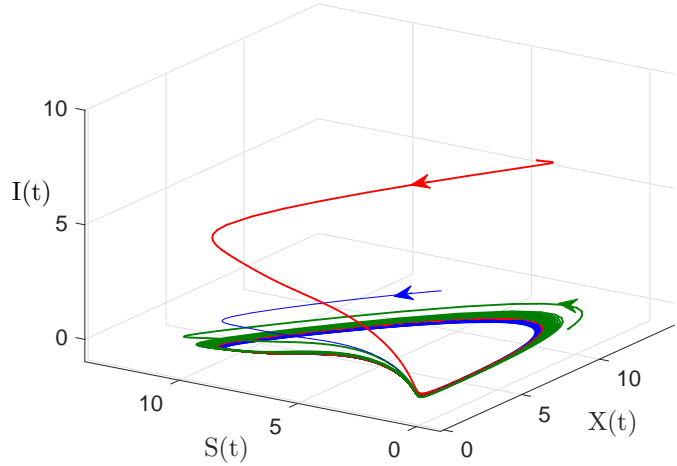}
\caption{Supercritical bifurcation. The behavior of the solution for different 
initial conditions and $\tau_1=16$. Phase portrait converges 
to the same limit cycle.}\label{fig:7}
\end{figure}
\begin{figure}[H]
\centering
\includegraphics[scale=0.7]{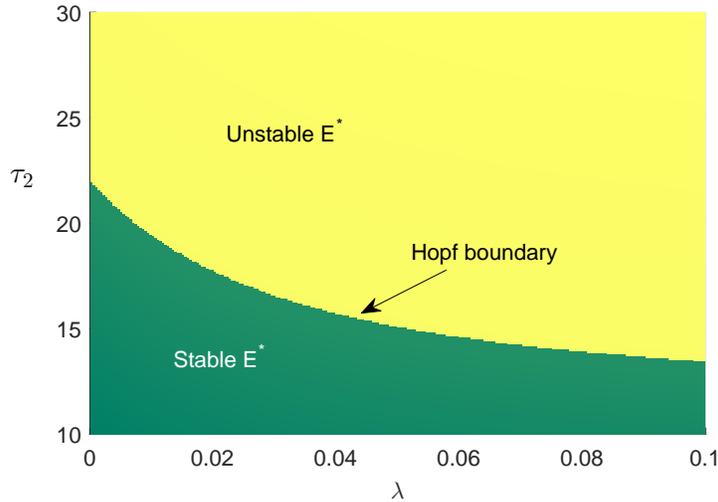}
\caption{Region of stability in the $\lambda$--$\tau_2$ parameter plane. 
Other parameters are as in Table~\ref{tab3.1}.}\label{fig:8}
\end{figure}


\subsection{Numerical simulations of optimal control}
\label{sec:ns:oc}
	
Finally, we solve numerically the optimality system 
(Theorem~\ref{thm:PMP:ourOCP}) and we display the
results found graphically. The parameters of model \eqref{state1} and \eqref{costate2}
are given in Table~\ref{tab3.1} and the initial values for \eqref{state1}
chosen as in \eqref{initial}. The weight constants of the objective functional
are selected, for illustrative purposes, as $B=\frac{1}{2}$, $B_2=\frac{1}{2}$ and $C=1$.
The optimal system \eqref{state1} has been solved numerically and the results
have been displayed diagrammatically. As already remarked at the end of
Section~\ref{sec:5}, this optimal system is a two-point boundary value problem (BVP)
with separated boundary conditions at times $t = 0$ and $t = t_f$.
An efficient technique to solve two-point BVPs numerically is the
collocation method \cite{hassan,MyID:433}. A suitable collocation code 
is given by solver \texttt{bvp4c} of the \textsf{MATLAB} numerical 
computing environment, which can be used to solve nonlinear two-point BVPs. 
Solving our system requires an iterative scheme developed by \cite{hattaf}.
This involves the use of an appropriate algorithm.
Let $h>0$ be the discretization step size and let us consider
integers $(n,m)\in\mathbb{N}^2$ with $\tau=mh$ and $t_f=nh$.
In our programming we consider $m$ knots to the left of $0$
and to the right of $t_f$, and we obtain the following partition:
\[
\Delta=(t_{-m}=-\tau<\cdots<t_{-1}<0<t_1<\cdots<t_n=t_f<\cdots< t_{n+m}).
\]
Then, we have $t_i=ih$, $-m\leq i\leq n+m$.
Next we define the state and adjoint variables
$X(t)$, $S(t)$, $I(t)$, $A(t)$, $\xi_1(t)$, $\xi_2(t)$, $\xi_3(t)$, $\xi_4(t)$, and control $u(t)$
in terms of the nodal points. Now, a combination of forward and backward difference approximations
is applied to obtain the solutions. To demonstrate the consequences numerically, we select the set 
of parameters of Table~\ref{tab3.1}.
	
The optimal control function $u^*=(u_1^*,u_2^*)$ 
is intended in such a way that it minimizes the cost functional 
given by \eqref{1.6}. In Figure~\ref{fig:9}, we present the numerical solutions of the population 
class with control. We operate the control through insecticide spraying up to $100$ days.
Crop biomass increased and pest population decreased significantly with the influence of the optimal 
profiles of global awareness and awareness based control activity, as shown in Figure~\ref{fig:9}. 
It is also observed that susceptible pest population goes to extinction within
first $90$ days due to effort of optimal control. Thus, optimal control by means of awareness-based 
biocontrol is of great value in controlling a pest problem is a crop field.

The Pontryagin controls are shown in Figure~\ref{fig:10}.
We see they are bang-bang controls.

\begin{figure}[H]
\centering
\includegraphics[scale=0.67]{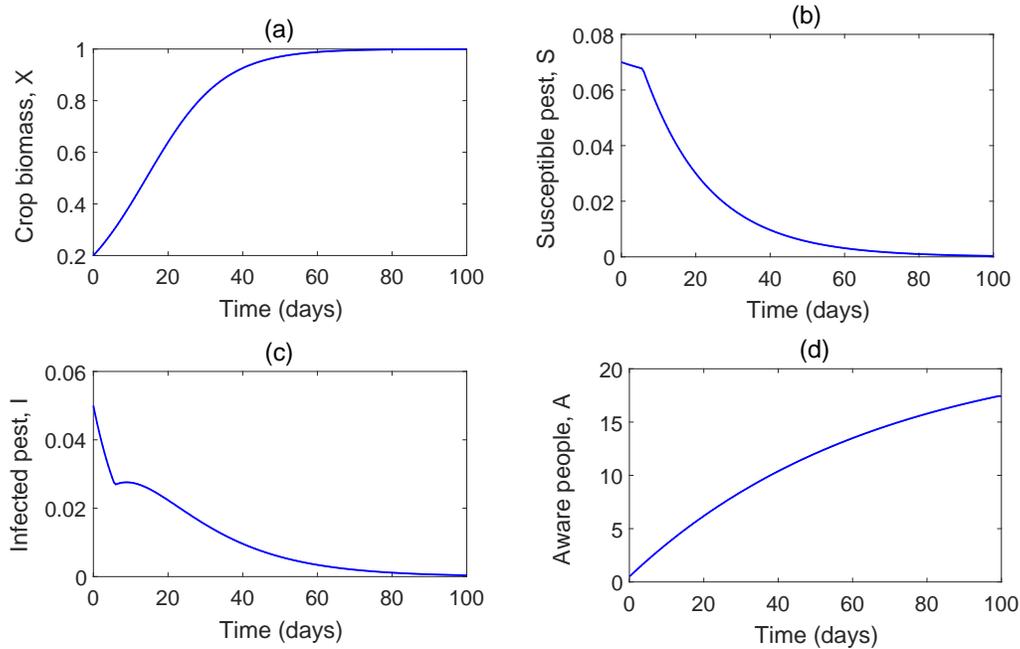}
\caption{Numerical solution of the optimal control problem with
$\tau_1=0$, $\tau_2=40$, and remaining parameter values from Table~\ref{tab3.1}.}
\label{fig:9}
\end{figure}
\begin{figure}[H]
\centering
\includegraphics[scale=0.68]{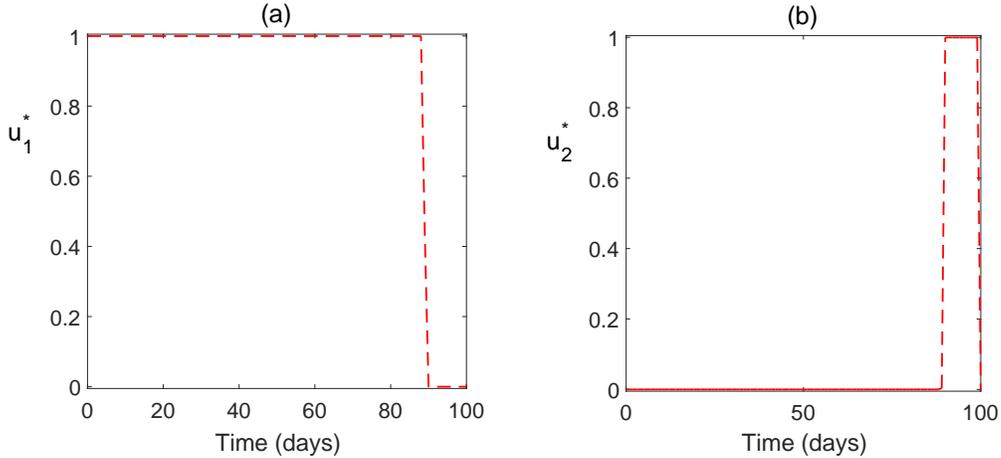}
\caption{Numerical approximation of the 
Pontryagin extremal controls $u_1^*$ and $u_2^*$ \eqref{1.13}.}
\label{fig:10}
\end{figure}
	
	
\section{Discussion and Conclusion}
\label{sec:7}
	
In this paper, a mathematical model has been proposed and analyzed
to study the effect of awareness programs on the control of pests
in agricultural practice. Namely, the model has for populations
the density of crop biomass, density of susceptible pests,
infected pests, and self-aware population. We suppose that self-aware
people adopt biological control as integrated pest management as it is
eco-friendly and less injurious to human health. With this approach,
the susceptible pest is made infected. Also, we presume that infected pests
cannot harm the crop. For this, awareness campaigns are taken as relative
to the density of susceptible pests present in the crop field. There might
be a delay in measuring the density of susceptible pests or in organizing
awareness programs. Thus, we have included time delays in the modeling process.
	
The proposed model exhibits three equilibria: (i) the origin, which is always unstable,
(ii) the pest-free equilibria, which are stable if the basic reproduction number $R_0$
is less than one while for $R_0 > 1$ it becomes unstable, (iii) the endemic equilibrium,
which exists and is, as well, stable for $R_0 < 1$. From our analytical and numerical study,
we observed that the most significant parameters involved in the system are $a$ and $\tau$, respectively
the awareness population growth rate and the combined time delay. If the impact of awareness campaigns increases,
the density of crop increases.  As a consequence, pest prevalence declines. On the other hand,
endemic equilibrium is locally stable when the delay is less than its critical value, i.e.,
when $\tau < \tau_0=44.5$, approximately. The system loses its stability when $\tau > \tau_0$
and a Hopf-bifurcation occurs at $\tau = \tau_0$. In conclusion, raising awareness among people
with tolerable time delay may be a good aspect for the optimal control of a pest in the crop field,
decreasing the serious issues that pesticides have on human health and surroundings.

It remains open the question of how to define a Lyapunov function 
to prove global stability for the proposed delayed model.
In this direction, the techniques of \cite{Khalid-Hattaf}
should be helpful. This is under investigation and will
be addressed elsewhere.
	
	
\section*{Acknowledgements}
	
Abraha acknowledges Adama Science and Technology University
for its hospitality and support during this work through
the research grant ASTU/SP-R/027/19.
Torres is grateful to the financial support from the
Portuguese Foundation for Science and Technology
(FCT), through CIDMA and project UIDB/04106/2020.
The authors are very grateful to two anonymous reviewers
for valuable remarks and comments, which significantly 
contributed to the quality of the paper.
	


\end{document}